\documentclass[10pt]{amsart}

\usepackage[a4paper]{geometry}
\usepackage[usenames, dvipsnames]{color}

\usepackage{graphicx}
\usepackage{amsmath,amsthm, amsfonts, amssymb}

\newtheorem*{theorem*}{Theorem}
\newtheorem{lemma}{Lemma}[section]
\newtheorem{proposition}[lemma]{Proposition}
\newtheorem{corollary}{Corollary}[section]
\theoremstyle{definition}
\newtheorem{definition}[lemma]{Definition}
\theoremstyle{remark}
\newtheorem{remark}{Remark}[section]
\newtheorem{example}{Example}[section]

\newcommand{\real} {\mathbb{R}}
\newcommand{\enteros} {\mathbb{Z}}
\newcommand{\disco} {\mathbb{D}}
\newcommand{\R}{{\bf\sf R}}
\newcommand{\subR}{{\bf\sf Q}}

\newcommand{\K}{{\bf\sf K}}

\newcommand{\dist}{\mbox{dist}}
\newcommand{\diam}{\mbox{diam}}

\title[Dynamical indicators and nonuniformly hyperbolic dynamics]{On the approximation of dynamical indicators \\ in systems with nonuniformly 
hyperbolic behavior}

\author{ 
Fernando Jos\'e S\'anchez-Salas}
\address{
Departamento de Matem\'aticas, Facultad Experimental de Ciencias, Universidad del Zulia, Avenida Universidad, Edificio Grano de Oro, Maracaibo, Venezuela
}
\email{fjss@fec.luz.edu.ve}

\date{March 25, 2015}



\thanks{The author thanks the Abdus Salam International Centre for Theoretical Physics (ICTP) for its hospitality during the preparation of this manuscript. I would also 
like to express my gratitude to professor Stefano Luzzatto for advice and encouragment. This work was partially supported by the Associateship 
Programme of ICTP}

\begin{document}

\begin{abstract}
Let $f$ be a $C^{1+\alpha}$ diffeomorphism of a compact Riemannian manifold and $\mu$ an ergodic hyperbolic measure with positive entropy. We prove that 
for every continuous potential $\phi$ there exists a sequence of basic sets $\Omega_n$ such that the topological pressure $P(f|\Omega_n,\phi)$ converges to the free 
energy $P_{\mu}(\phi) = h(\mu) + \int\phi{d\mu}$. We also prove that for a suitable class of potentials $\phi$ there exists a sequence of 
basic sets $\Omega_n$ such that $P(f|\Omega_n,\phi) \to P(\phi)$.  
\end{abstract}

\maketitle

\section{Introduction}\label{sec:introduction}

This is work is concerned with the approximation of dynamical indicators in systems with nonuniformly hyperbolic behavior.

(Uniformly) hyperbolic dynamics is characterized by a (continuous) decomposition of the tangent space $T_xM$ into invariant subspaces which are contracted 
(resp. expanded) by the derivative. The local instability of the orbits generated by this structure and the recurrence due to the compactness of the space gives rise 
to a complex and very rich orbit structure which is well understood. Among other things uniformly hyperbolic systems exhibit strong recurrence and mixing properties, 
many invariant measures, positive entropy and abundance of periodic points. Moreover, they are robust and structurally stable and can be modelled by Markov chains both 
topologically and from the measure-theoretical point of view. See \cite{hasselblat} and \cite{katok.hasselblat} for a comprehensive presentation of the theory.

Nonuniformly hyperbolic dynamical systems were introduced by Pesin in the early seventies as a generalization the notion of uniformly hyperbolic dynamics. Invariant measures 
are at the heart of the theory of nonuniformly hyperbolic systems. We say that an $f$-invariant Borel probability $\mu$ is \emph{hyperbolic} if all its Lyapunov exponents
\begin{equation}\label{lyapunov.exponents}
\chi(x,v) = \lim_{n \to \pm\infty}\dfrac{\|Df^n(x)v\|}{n}, \quad\forall \ v \in T_xM-\{0\}
\end{equation}
are nonzero $\mu$-a.e. By Oseledec's theorem if $\mu$ is hyperbolic then there exists an $f$-invariant Borel subset $K \subset M$ and a splitting 
into stable $E^s_x$ and unstable $E^u_x$ Borel measurable fields of subspaces --in opposition to continuous-- of the tangent space over $K$. Vectors in $E^s_x$ (resp. $E^u_x$) are 
\emph{asymptotically} contracted (resp. expanded) by the derivative, that is, the time $N$ needed for every vector $v \in E^s_x$ (resp. $v \in E^u_x$) to be contracted (resp. expand) 
depends on $x$ in a very irregular way --typically as a Borel function of the point-- moreover the angle $\angle(E^s_x,E^u_x)$ is a Borel function of $x$ and decays 
to zero with subexponential rates along the orbits. This set of conditions define the notion of nonuniform hyperbolicity. 
See subsection \ref{subsec:pesin.theory} for details.

We refer to \cite{barreira.pesin} for an up-to-date overview of the theory and to \cite{pesin.2010} for a survey on open problems on nonuniformly 
hyperbolic dynamical systems.

The following well known theorem due to A. Katok and L. Mendoza is a sample of the type of results that we are interested:
\ 
\\
\\
{\em 
Let $f$ be a $C^{1+\alpha}$ diffeomorphism of a compact Riemannian manifold $M$ and $\mu$ be a hyperbolic measure with positive metric entropy. Suppose in addition that 
$\mu$ is ergodic. Then there exists a sequence of hyperbolic horseshoes $\Omega_n$ and ergodic measures $\mu_n$ supported on 
$\Omega_n$ such that:
\begin{itemize}
\item $\mu_n \to \mu$, in the weak$^*$ topology and
\item $h(\mu_n) \to h(\mu)$.
\end{itemize}
}
\ 
\\
See \cite[Theorem S.5.10]{katok.mendoza}. 

A. Katok laid down the foundations to study this type of problems in his seminal paper \cite{katok} about relations between entropy, 
periodic orbits and Lyapunov exponents of systems with nonuniformly hyperbolic behavior. More recently these questions received attention in \cite{gelfert.wolf}, 
\cite{gelfert.2009}, \cite{gelfert.2010}, \cite{liang.liu.sun}, \cite{luzzatto.sanchez}, \cite{sanchez-salas.1}, \cite{sanchez-salas.2} and \cite{wang.sun}.

Katok-Mendoza's theorem suggests to ask whether or not it is possible to approximate, along suitable sequences of hyperbolic sets, dynamical indicators 
such as topological pressure, fractal dimensions and Lyapunov exponents, in systems with nonuniformly hyperbolic behavior.

In this note we make the case for $P(\phi)$, the topological pressure of a continuous potential $\phi$. This quantity is a topological invariant of 
the dynamics which generalizes the notion of topological entropy and can be defined as a weighted rate of growing of the number of finite, dynamically non equivalent 
orbits, up to finite precision. A central result in the thermodynamic formalism is the following {\em variational principle},
\begin{equation}\label{additive.variational.principle}
P(\phi) = \sup_{\mu \in {\mathcal M}_f}\left\{h(\mu) + \int\phi{d\mu}\right\},
\end{equation}
where $h(\mu)$ denotes the {\em Kolmogorov-Sinai entropy}, $P_{\mu}(\phi) := h(\mu) + \int\phi{d\mu}$ is the free energy or measure-theoretical pressure and 
${\mathcal M}_f$ the set of $f$-invariant Borel probabilities. See \cite{bowen} and \cite{katok.hasselblat}.

This notion plays a central role in the ergodic theory of systems with some hyperbolicity as long as some valuable information about Lyapunov exponents, 
fractal dimensions, multifractal spectra and invariant measures which are extreme points of certain variational principles can be extracted from the topological pressure of 
suitable potentials. See \cite{barreira.2010} and \cite{barreira.gelfert}. 

Our first result is a generalization of Katok-Mendoza's theorem for the free energy of a continuous potential $\phi$ with respect to a hyperbolic measure $\mu$ with 
positive entropy.

Let $\mu_n$ be the sequence of hyperbolic measures given by Katok-Mendoza's theorem. Then, for every continuous $\phi$
$$
h(\mu_n) + \int\phi{d\mu_n} \to h(\mu) + \int\phi{d\mu},
$$
and therefore, by the variational principle (\ref{additive.variational.principle}),
$$
P(\phi) \geq \limsup_{n \to +\infty}P(f|\Omega_n,\phi) \geq \liminf_{n \to +\infty}P(f|\Omega_n,\phi) \geq h(\mu) + \int\phi{d\mu}.
$$
We prove that $\Omega_n$ can be chosen carefully in such way that the limit exists and it is equal to the measure-theoretical pressure $P_{\mu}(\phi)$.
\ 
\\
\\
{\bf Theorem A} \ {\em Let $f$ be a $C^{1+\alpha}$ diffeomorphism of a compact manifold and $\mu$ an ergodic, hyperbolic measure with $h(\mu) > 0$. Then, for every 
continuous potential $\phi$ there exists a sequence of basic sets $\Omega_n = \Omega_n(\mu,\phi)$ with rate 
of hyperbolicity $\chi(\Omega_n) \geq \beta$ for some constant $\beta > 0$ only depending on $\chi(\mu)$, the rate of hyperbolicity of $\mu$, such that
\begin{equation}\label{main.1}
P(f|\Omega_n,\phi) \to h(\mu) + \int\phi{d\mu}.
\end{equation}
Furthermore $\Omega_n$ has the following {\bf strong approximation property}: $\mu_n \to \mu$ for every sequence of ergodic measures $\mu_n \in {\mathcal M}_f(\Omega_n)$.
}
\ 
\\
\\
Recently Gelfert announces a similar result in \cite{gelfert.2014} for $C^{1+\alpha}$ diffeomorphisms or $C^1$ diffeomorphisms preserving a 
hyperbolic $f$-invariant probability having a dominated splitting on its support. We recall that the \emph{rate of hyperbolicity of an 
Oseledec's regular point $x$} as
$$
\chi(x) := \min\{|\chi(x,v)| : v \in T_xM-\{0\}\}.
$$
Given an $f$-invariant Borel probability $\mu$, we define the \emph{rate of hyperbolicity of $\mu$} as
$$
\chi(\mu) := \int\chi(x)d\mu(x)
$$ 
Also, given a compact $f$-invariant subset $\Omega$ we define its rate of hyperbolicity as
$$
 \chi(\Omega) := \inf\{\chi(\mu): \mu \in \mathcal{M}_f(\Omega)\},
$$
where $\mathcal{M}_f(\Omega)$ is the set of $f$-invariant probabilities in $\Omega$

The probabilities $\mu_n$ provided by Katok-Mendoza's theorem are measures of maximal entropy, that is, $h(\mu_n) = h_{top}(f|\Omega_n)$. Therefore if 
there would exists a sequence of ergodic hyperbolic measures $\mu_n$ with positive entropy such that $h(\mu_n) \to h_{top}(f)$ then, by an easy 'diagonal' argument 
we can find a sequence of hyperbolic horseshoes $\Omega_n$ such that 
$$
h_{top}(f|\Omega_n) \to h_{top}(f).
$$
Of course, a good amount of hyperbolicity in the phase space is necessary for this type of approximation results. Following this idea one 
may ask whether or not there exists in systems with sufficient hyperbolicity, a sequence of hyperbolic sets $\Omega_n$ such that
$$
P(f| \Omega_n,\phi) \to P(\phi),
$$
for every continuous $\phi$. However, the following example shows that the answer to this question is, in general, negative, even if the system is nonuniformly hyperbolic. 

\begin{example}
Let $\Omega_0 \subset \disco^2$ be a horseshoe with internal tangencies defined inside a compact disc $\disco \subset \real^2$. This is a nonuniformly 
hyperbolic set with positive topological entropy. See \cite{cao.luzzatto.rios}. Let us plug $\Omega_0$ as a compact $f$-invariant subset of a $C^{\infty}$ diffeomorphism of the sphere in the usual 
way with a source at the north-pole $N$ and a sink at the south-pole $S$. Notice that $h_{top}(f) = h_{top}(f | \Omega_0)$ and that every $f$-invariant basic set 
$\Omega$ is contained in $\Omega_0$. Let $\phi$ be a continuous function such that $\phi(N) = 2h_{top}(f)$ and $\phi(x) = 0$, for every $x \not\in U$, where $U$ is small 
neighborhood of $N$ contained in the connected component of $W^u(N)$ containing $N$. Then every point ergodic $f$-invariant Borel probability is either $\delta_N$, an 
ergodic measure supported on $\Omega$ or $\delta_S$, the Dirac measure concentrated at the south-pole. Then,
$$
P(\phi) = \sup_{\mu \in \mathcal{M}_f}\{h(\mu) + \int\phi{d\mu}\} = \phi(N) = 2h_{top}(f)
$$
and, for every $f$-invariant basic set $\Omega \subset S^2$,
$$
P(\phi) >  P(f|\Omega,\phi).
$$
\end{example}
 
This happens since the support of $\phi$ is away from the part of phase space where basic sets are located. Therefore if $\phi$ captures the hyperbolicity 
of the phase space then it would be possible to approximate $P(\phi)$ by suitable sequences of hyperbolic sets. 

\begin{definition}\label{definition.hyperbolic.potential}
We say that $\phi$ is a \emph{hyperbolic potential} if 
$$
P(\phi) - \sup_{\mu \in \mathcal{M}_f}\int\phi{d\mu} > 0
$$
and there exists a sequence of ergodic hyperbolic measures $\mu_n$ such that
$$
h(\mu_n) + \int\phi{d\mu_n} \to P(\phi).
$$
We denote by $C(\mathcal{H})$ the set of hyperbolic potentials.
\end{definition}
\ 
\\
{\bf Theorem B}
{\em Let $f$ be a $C^{1+\alpha}$ diffeomorphism of a compact manifold with $h_{top}(f) > 0$. Then, for every $\phi \in C(\mathcal{H})$ there exists a sequence $\Omega_n$ 
of basic sets such that 
$$
P(f|\Omega_n,\phi) \to P(\phi).
$$
In particular, it holds the following variational equation:
\begin{equation}\label{main.2}
P(\phi) = \sup_{\Omega \in {\mathcal H}}\,P(f | \Omega,\phi),
\end{equation}
where $\mathcal{H}$ is the family of $f$-invariant basic sets. 
}
\ 
\\
\\
Hyperbolic potentials were introduced in \cite{gelfert.wolf} by K. Gelfert and C. Wolf. There they proved that the topological pressure 
of these potentials can be computed as a weighted rate of growing of hyperbolic periodic orbits filtrated according to the quality of its hyperbolicity.
\ 
\\
\\
{\bf Proof of Theorem B:} let $\phi \in C(\mathcal{H})$ be a hyperbolic potential, the existence of a sequence of approximating basic sets $\Omega_n$ for $P(\phi)$ 
follows from Theorem A by the following straightforward 'diagonal' argument: let $\phi \in C(\mathcal{H})$ and $\mu_n$ be a sequence of hyperbolic $f$-invariant ergodic 
probabilities such that $h(\mu_n) + \int\phi{d\mu_n} \to P(\phi)$. Then $h(\mu_n) > 0$ for every sufficiently large $n$. Indeed, let 
$0 < \epsilon < P(\phi) - \sup_{\mathcal{M}_f}\int\phi{d\mu}$. Then, 
$h(\mu_n) + \int\phi{d\mu_n} > P(\phi) - \epsilon$ for every large $n$. Therefore,
$$
h(\mu_n) > P(\phi) - \int\phi{d\mu_n} - \epsilon \geq P(\phi) - \sup_{\mu \in \mathcal{M}_f}\int\phi{d\mu} - \epsilon > 0.
$$
By Theorem A, for each $\mu_n$ there exists a sequence of basic sets $\Omega^m_n$ such that $P(f|\Omega^m_n,\phi)$ converges to the free energy 
$h(\mu_n) + \int\phi{d\mu_n}$. Passing to a suitable 'diagonal' sequence $\Omega_n = \Omega^{m_n}_n$ we get a sequence of basic sets 
such that $P(f|\Omega_n,\phi) \to P(\phi)$. The variational equation (\ref{main.2}) holds since $P(\phi) \geq P(f|\Omega,\phi)$ for every compact $f$-invariant subset. \ {\bf QED}

Theorem A is a consequence of the methods that we developed with S. Luzzatto in \cite{luzzatto.sanchez} and an idea of Mendoza in \cite{mendoza.1988}. Of course hyperbolic horseshoes 
in Katok-Mendoza's theorem \cite[Theorem S.5.10]{katok.mendoza} are basic sets. However our construction differ in several points from \cite{katok.mendoza}. In 
particular, our approximating sets are $f$-invariant saturate of horseshoes with finitely many branches with variable return time (see \ref{subsec:aleksev.sets}) with 
the strong approximation property mentioned at Theorem A. 

We start observing that the measure-theoretical pressure  $P_{\mu}(\phi) = h(\mu) + \int\phi{d\mu}$ is a weighted rate of growing of dynamically 
non equivalent finite typical orbits, up to finite precision (see Proposition \ref{free.energy.prop} in section \ref{sec:proof.main.technical.lemma.1}). Compare 
\cite[Theorem 1.1]{mendoza.1988}. Then we draw carefully finitely many finite orbits which are a good sample for this statistic with small precision. These orbits return to a suitable non invariant 
uniformly hyperbolic set or Pesin set giving rise to \emph{hyperbolic branches} $f^{n_i}: S_i \to U_i$ with variable return times $n_i$, 
where by hyperbolic branches we mean diffeomorphisms mapping 'vertical' strips $S_i$ onto 'horizontal' strips $U_i$ 
inside a fixed rectangle crossing each other transversally according to Smale's horseshoe model, contracting in the vertical direction and 
expanding distances in the horizontal. See subsection \ref{subsec:aleksev.sets} and definition \ref{definition.hyperbolic.branch}. 
This is a consequence of the pseudo-Markov property of coverings of the hyperbolic Pesin sets by regular Lyapunov rectangles (see Proposition \ref{markov.covers}) in 
subsection \ref{subsec:pesin.theory}. Let $\Omega^*$ be the maximal invariant subset of the piecewise smooth map $F$ defined by the hyperbolic branches so chosen. $\Omega^*$ is endowed with a 
hyperbolic product structure according to \cite[Definition 1]{young}, that is, two transversally intersecting continuous laminations $\mathcal{F}^s$ and $\mathcal{F}^u$ with an angle bounded from below 
which are contracted (resp. expanded) exponentially by iterations of $F$ and such $\Omega^* = \bigcup\mathcal{F}^s \cap \bigcup\mathcal{F}^u$. 
See subsection \ref{subsec:aleksev.sets}.

The hyperbolic branches $f^{n_i}: S_i \to U_i$ so chosen are quasi-generic meaning that the iterates of every point $x \in S_i$ 
gives a good approximation of $\mu$, up to a small precision. Then every ergodic measure supported on $\Omega$, the $f$-invariant saturate of $\Omega^*$, is near to 
$\mu$ in the weak topology.

Then we prove that $P(f|\Omega,\phi)$ is good approximation of the measure-theoretical pressure $P_{\mu}(\phi)$ by estimating the topological pressure of 
$\Omega$ as a weighted rate of growing of hyperbolic periodic orbits in $\Omega$,
\begin{equation}\label{weighted.hyperbolic.counting}
 P(f|\Omega,\phi) = \limsup\limits_{n \to +\infty}\dfrac{1}{n}\log\left(\sum_{x \in Per_n(f|\Omega)}\exp\sum_{j=0}^{n-1}\phi(f^j(x))\right).
\end{equation}

For this we use a shadowing argument to compare the weight of the periodic orbits of $\Omega$ with the weight of the chosen sample of finitely many returning 
points generating $\Omega^*$. See section \ref{sec:proof.main.technical.lemma.3}. Here some care has to be taken to keep track of the combinatorics of periodic orbits, 
due to the variable return times defining $\Omega^*$. This is done in section \ref{sec:proof.main.technical.lemma.2}.
\ 
\\
\\
\textbf{Organization of the paper}. The paper is organized as follows: Section \ref{sec:proof.main.technical.lemma.preliminaries} contain 
preliminary material to the proof of Theorem A: in subsections \ref{subsec:aleksev.sets} and \ref{subsec:uniform.approximation.property} we recall the notion of an Alekseev set and the uniform approximation 
property. In subsection \ref{subsec:pesin.theory} we recall main definitions of nonuniformly hyperbolic dynamics, Pesin sets and pseudo-Markov property used
in the construction of Aleeksev sets. We give the arguments to choose $\Omega$ at section \ref{sec:proof.main.technical.lemma.1}. 
The estimation of the topological pressure is done at sections \ref{sec:proof.main.technical.lemma.2} and \ref{sec:proof.main.technical.lemma.3}.

\section{Proof of Theorem A: preliminaries}\label{sec:proof.main.technical.lemma.preliminaries}

\subsection{The geometrical model: Alekseev sets}\label{subsec:aleksev.sets}

Our geometric model will be defined by a finite collection $\mathcal S$ of pairwise disjoint 
{\em stable cylinders} $\{S_{1},.., S_{N}\}$ and corresponding pairwise disjoint collection $\mathcal U$ of 
{\em unstable cylinders} \ $\{U_{1},..., U_{N} \}$ contained in a rectangle $\R$ 
which are the domain (resp. co-domain) of suitable {\em hyperbolic branches}
$$
f^{R_{i}}: S_{i}\to U_{i}
$$
defined by finitely many return times \( R_{i} \). By hyperbolic we mean that $f^{R_i}$ contracts (resp. expands) in the 'vertical' 
(resp. 'horizontal') directions; 
that is, they preserve suitable continuous families of cone. See Definition \ref{definition.hyperbolic.branch} in subsection 
\ref{subsec:pesin.theory} for details.

\begin{definition}
An {\em Alekseev set} \ \( \Omega^{*}\) is defined by an array of hyperbolic branches $\{f^{R_{i}}: S_{i}\to U_{i}\}$ all whose stable 
cylinders 
\( S_{i} \) {\em 'crosses'} all \( U_{i} \)'s transversally and such that every \( U_{i} \) {\em 'crosses'} all \( S_{i} \)'s transversally. 
\( \Omega^{*}\) is the maximal invariant set in $\R$ under iterations of \( f^{R} \) and its inverse
$$
\Omega^* := \bigcap\limits_{n \in \enteros}(f^R)^{n}(\R),
$$
where $f^{R} : \bigcup_iS_i \to \bigcup_iU_i$ is the piecewise smooth invertible map define by 
\[
f^{R}|_{S_{i}}:= f^{R_{i}}|_{S_{i}} \quad\text{ and } \quad (f^{R})^{-1}|_{U_{i}}:= f^{-R_{i}}|_{U_{i}}.
\]
\end{definition}

\begin{remark}
This construction was originated in the work of M. V. Alekseev aiming at to describe topological analogues of Markov chains. See 
\cite{alekseev}.
\end{remark}

The next couple of technical results were proved in \cite[Section 3]{luzzatto.sanchez}. 

\begin{lemma}\label{alekseev.set}
$\Omega^*$ is an $f^{R}$-invariant Cantor set endowed with a hyperbolic product structure by which we mean two continuous laminations of 
local $f^{R}$-invariant 
manifolds $\mathcal F^{S}$ (resp. $\mathcal F^{U}$) with angles uniformly bounded from below by a constant $ > 0$ which are exponentially 
contracted (resp. expanded) 
by $f^{R}$ and such that
$$
\Omega^* = \left(\bigcup\mathcal F^{S}\right) \cap \left(\bigcup\mathcal F^{U}\right).
$$ 
\end{lemma}

These hyperbolic Cantor sets are the primary blocks in our construction of approximating basic sets.

\begin{proposition}
Let $\Omega^*$ an Alekseev set defined by finitely many hyperbolic branches $f^{R_{i}}: S_{i}\to U_{i}$, then:
\begin{equation}
\Omega = \bigcup_i\bigcup_{j=0}^{R_i-1}f^j(\Omega^*_i),
\end{equation}
is the $f$-invariant saturate of $\Omega^*$ and it is a topologically transitive, locally maximal, uniformly hyperbolic $f$-invariant subset. 
\end{proposition}

See \cite[Lemma 3.1]{luzzatto.sanchez} and references therein.

\subsection{Strong approximation property}\label{subsec:uniform.approximation.property}

The sequences $\Omega_n$ in Theorem A approximate uniformly $\mu$ in that $\mu_n \to \mu$ for every sequence $\mu_n$ of ergodic measures 
such that 
$supp \mu \subset \Omega_n$. Actually, given an open neighborhood ${\mathcal N}$ of $\mu$ in the weak-* topology our methods allows to 
construct hyperbolic basic sets 
$\Omega = \Omega({\mathcal N})$ such that $\nu \in {\mathcal N}$ for every ergodic Borel probability $\nu$ supported on $\Omega$. This is 
done as follows.

Let $\mu$ be a Borel probability satisfying our main hypotheses. First recall that a point $x$ is generic for $\mu$ if
\[ 
\frac 1n\sum_{j=0}^{n-1}\phi(f^{j}(x)) \to \int\phi{d\mu} \quad\text{as}\quad  n\to \infty \quad\text{for all continuous functions} \ \phi \in C^0(M).
\]

Given a countable dense subset $\{\psi_i\}$ of $C^0(M)$ we denote, given two constants $\rho, s > 0$, the weak-\(*\) \( (\rho, s) \) neighborhood of 
\( \mu \)
\begin{equation}
{\mathcal O}(\mu, \rho, s) := \{\nu : \left|\int\psi_i{d{\mu}} - \int\psi_i{d{\nu}}\right| < \rho, \ i = 1, \cdots , s\};
\end{equation}
Clearly, $\mu_n \to \mu$ in the weak-* topology if and only if there are sequences $\rho_n \to 0^+$ and $s_n \to +\infty$ such that 
$\mu_n \in {\mathcal O}(\mu, \rho_n, s_n)$.

\begin{definition}\label{quasigeneric.definition}
We say that a point \( x \) is \( (\rho, s, n) \) \emph{quasi-generic} for the measure \( \mu \) if 
\[ 
\left|\frac 1n\sum_{j=0}^{n-1}\phi_{i}(f^{j}(x))-\int\phi_{i}d\mu\right|\leq \rho \quad \forall i\leq s.
\]

Furthermore, we say that a hyperbolic branch 
\[
f^n : S \to U
\]
is {\em $(\rho,s)$-quasi-generic} for \( \mu \) if every \( x\in S \) is \( (\rho, s, n )  \) quasi-generic for \( \mu \).
\end{definition}

We underline that to be  \(\rho, s, n) \)-quasi-generic simply means that the empirical measure $\mathcal{E}_{x,n} = 1/n\sum_{k=0}^{n-1}\delta_{f^k(x)}$ 
belongs to $\mathcal{O}(\mu,\rho,s)$.

\begin{proposition}\label{strong.approximation.property}
Let \( \rho, s > 0 \) and suppose there exists an Alekseev set \( \Omega^{*}(\rho, s) \) defined by \( (\rho,s) \) quasi-generic branches. 
Then $\mu_{\Omega} \in {\mathcal O}( \mu, 3 \rho, s)$ for every  \( f \)-invariant ergodic probability measure $\mu_{\Omega}$ supported on $\Omega(\rho,s)$, 
the $f$-invariant saturate of $\Omega^{*}(\rho, s)$. In particular,
$$
{\mathcal M}_f(\Omega(\rho,s)) \subset {\mathcal O}( \mu, 3 \rho, s),
$$
where ${\mathcal M}_f(\Omega(\rho,s))$ denotes the set of $f$-invariant Borel probabilities supported on $\Omega(\rho,s)$. 
\end{proposition}

We refer the reader to \cite{luzzatto.sanchez} for details.

\subsection{Nonuniform hyperbolicity and pseudo-Markov coverings}\label{subsec:pesin.theory}

In this section we collect some facts on nonuniformly hyperbolic dynamics necessary to prove Theorem A.

Let $f$ be a $C^r$ $r \geq 1$ diffeomorphism of a compact manifold $M$. We say that a point $x$ is \emph{Oseledec regular} if there 
exists numbers $\chi_1(x) < \cdots < \chi_{k(x)}(x)$ and a sum direct decomposition into subspaces $T_xM = \bigoplus_{i=1}^{k(x)}{E_i(x)}$ 
such that
$$
\lim_{n \to \pm\infty}\dfrac{\log\|Df^n(x)v\|}{n} = \chi_i(x) \quad\forall \ v \in E_i(x) - \{0\}.
$$
Notice that if $x$ is regular then so is $f^n(x)$ for every $n \in \enteros$ and therefore we can speak of a regular orbit.

Oseledec's theorem (see \cite{barreira.pesin} and \cite{katok.mendoza}) says that the set of \emph{regular points} $\Sigma$ is a Borel 
subset of total probability i.e. it has 
$\mu(\Sigma) = 1$, for every $f$-invariant Borel probability $\mu$. Moreover, the functions $\chi_i = \chi_i(x)$, $k = k(x)$, $E_i(x)$ are 
$f$-invariant and 
Borel measurable and the angle between the subspaces $E_i(x)$ decays subexponentially along the orbits, that is,
\begin{equation}
\lim_{n \to \pm\infty}\dfrac{\log\angle(E_S(f^n(x)),E_{N-S}(f^n(x))}{|n|} = 0
\end{equation}
for every finite subset $S \subset N := \{1, \cdots , k(x)\}$, where
$$
E_S(x) := \bigoplus_{i \in S}E_i(x).
$$

Given an \( f \)-invariant Borel probability hyperbolic measure \( \mu \) then the Lyapunov exponents \( \chi_i(x) \) are well defined for 
every 
\( x\in\Sigma \) and they are non-zero. Moreover, if \( \mu \) is ergodic there exists a constant $\chi $ 
satisfying
 \begin{equation}\label{eq:chi}
  \inf\{|\chi_i(x)|: x\in \Sigma, i = 1, \cdots , k(x)\}> \chi > 0 \quad\text{for}\quad \mu-a.e.
  \end{equation}
Then, for all  sufficiently small $\epsilon > 0$ such that \( \chi > \epsilon \) by Oseledec's theorem there exist measurable 
$Df$-invariant decompositions
$$
T_xM = E^s(x) \oplus E^u(x),
$$
and, for every $\epsilon > 0$, tempered Borel measurable functions $C_{\epsilon}, K_{\epsilon} : \Sigma \to (0,+\infty)$ with subexponential growth such that
$$
\begin{cases}
\|Df^n(x)v\|    \leq C_{\epsilon}(x)e^{n(-\chi+\epsilon)}\|v\|  & \ \forall \ v \in E^s(x) \ \forall \ n \geq 0 \\
\|Df^{-n}(x)v\| \leq C_{\epsilon}(x)e^{n(-\chi+\epsilon)}\|v\|  & \ \forall \ v \in E^u(x) \ \forall \ n \geq 0 \\
\end{cases}
$$
and $\angle(E^s(x),E^u(x)) \geq K_{\epsilon}(x)$, where
$$
E^s(x) := \bigoplus\limits_{\chi_i(x) < 0}E_i(x) \quad\text{and}\quad E^u(x) := \bigoplus\limits_{\chi_i(x) > 0}E_i(x).
$$
By tempered we mean slowly growing and/or decay, that is, 
$$
(1+\epsilon)^{-1} \leq \dfrac{C_{\epsilon}(f(x))}{C_{\epsilon}(x)}, \dfrac{K_{\epsilon}(f(x))}{K_{\epsilon}(x)} \leq 1+\epsilon, \quad\mu-a.e.
$$
This follows from the tempering-kernel lemma \cite[Lemma S.2.12]{katok.mendoza}.

We remark that the properties given above as a consequence of the hyperbolicity of \( \mu \) can also be formulated without any reference 
to the measure \( \mu \) and are essentially nonuniform versions of standard uniformly hyperbolic conditions, 
see \cite[Theorem 6.6]{barreira.pesin}. 

We refer the reader to \cite{barreira.pesin} and \cite{katok.mendoza} for an exposition of the ergodic theory of smooth dynamical systems 
with hyperbolic behavior.

We now introduce a standard ``filtration'' of \( \mu \) almost every point which gives a countable number of nested, 
uniformly hyperbolic (but not \( f \)-invariant) sets, often referred to as ``Pesin sets'', whose points admit uniform hyperbolic bounds 
and uniform 
lower bounds on the sizes of the local stable and unstable manifolds.

For  $\chi > 0$ as in \eqref{eq:chi} above, and every positive integer $\ell > 0$, we define a (possibly empty) 
compact (not necessarily invariant) set $\Lambda_{\chi,\ell} \subset M$ such that $E^s|_{\Lambda_{\chi,\ell}}$ and  $E^u |_{\Lambda_{\chi,\ell}}$ vary continuously with the point $x \in \Lambda_{\chi,\ell}$ and  such that 
$$
\begin{cases}
\|Df^n(x)v\| \leq {\ell}e^{-n\chi}\|v\| & \|Df^{-n}(x)v\| \geq \ell^{-1}e^{n\chi}\|v\| \ \forall \ v \in E^s(x) \ \forall \ n \geq 0 \\
\|Df^{-n}(x)v\| \leq {\ell}e^{-n\chi}\|v\| & \|Df^n(x)v\| \geq \ell^{-1}e^{n\chi}\|v\| \ \forall \ v \in E^s(x) \ \forall \ n \geq 0.
\end{cases}
$$
Moreover,  the angles between the stable and unstable subspaces satisfy
$$
\angle(E^s(x),E^u(x)) \geq \ell^{-1}
$$
for every  $x \in \Lambda_{\chi,\ell}$. As the rate of hyperbolicity of $\mu$ is bounded from below by $\chi > 0$ 
we have 
$$
\mu(\Lambda_{\chi,\ell}) \to 1 \quad\text{as}\quad \ell \to +\infty.
$$

\begin{definition}
We say that $\R(x)$ is a rectangle centered at $x$ if it is the image of an embedding $e_x : I^n \to M$ such that $e_x(0) = x$, where $I = [-1,1]$.
\end{definition}

By \cite[Theorem S.3.1]{katok.mendoza} for every $\epsilon > 0$ and for $\mu$-a.e. $x \in \Sigma$ there exists a 
local coordinate $\phi_x : B(0,\rho(x)) \to M$, named \emph{Lyapunov charts}, such that the representative $f_x := \phi_{f(x)}^{-1} \circ f \circ \phi_x$ of 
$f$ in the new coordinates is a small perturbation of a hyperbolic linear isomorphism $Df_x(0): \real^m \to \real^m$ preserving 
the decomposition $\real^m = \real^s \oplus \real^u$ such that:
\begin{itemize}
\item $D\phi_x(0)\real^i = E_x^i$ for $i = s,u$; 
\item $e^{\chi-\epsilon} \leq \|Df_x(0)|\real^s\|, \|(Df_x(0))^{-1}|\real^u\| \leq e^{\chi+\epsilon}$;
\item $f_x = Df_x(0) + h_x$ where $\|h_x\|_{C^1} = \sup_{z \in I^n}\max\{\|h_x(z)\|,\|Dh_x(z)\|\} < \epsilon$
\end{itemize}
This is consequence of the $C^{1+\alpha}$ hypotheses and the hyperbolicity of the orbit. Nouniformity is captured by the slowly fluctuations 
of the radius $\rho(x)$ along the orbit, i.e.
$$
\lim_{n \to \pm\infty}\dfrac{\log\rho(f^nx)}{|n|} = 0 \quad\mu-a.e.
$$
Let 
$[-t(x),t(x)]^m \subset B(0,\rho(x))$ the largest $m$-cube contained in $B(0,\rho(x))$ and let $\sigma_x : [-1,1]^m \to [-t(x),t(x)]^m$ 
a homothety. Then we introduce the modified Lyapunov chart
$$
e_x := \sigma_x \circ \phi_x : [-1,1]^m \to M.
$$

\begin{definition}
 We call $\R(x) := e_x([-1,1]^m)$, the image of the modified Lyapunov chart, a \emph{rectangle}. 
\end{definition}

Now we define admissible manifolds. They are good approximations to local stable and unstable invariant manifolds.

\begin{definition}
A {\em admissible stable manifold} is a graph $\gamma^s = \{e_x(z,\hat{s}(z)): z \in I^s\}$, 
where $\hat{s} \in C^1(I^s,I^u)$ is a smooth map with $Lip(\hat{s}) := \sup_{z \in I^s}\|D\hat{s}(z)\| \leq \gamma$. We define similarly 
{\em admissible unstable manifolds}: $\gamma^u = \{e_x(\hat{u}(z),z): z \in I^u\}$, 
where $\hat{u} \in C^1(I^u,I^s)$ has $Lip(\hat{u}) := \sup_{z \in I^u}\|D\hat{u}(z)\| \leq \gamma$
\end{definition}

Admissible manifolds endow $\R(x)$ with a product structure: any given pair of admissible manifolds $\gamma^s$ and $\gamma^u$ intersects 
transversally at 
a unique point with an angle bounded from below. Moreover, the map $(\gamma^s, \gamma^u) \to \gamma^s\cap\gamma^u$ so defined satisfies a 
Lipschitz 
condition  \cite[\S 3.b]{katok.mendoza} and \cite [\S 8]{barreira.pesin}. The transversal structure of the admissible stable and unstable 
manifolds inside 
a rectangle \( \R \) allows us to define the notion of {\em admissible stable and unstable cylinders}. 
An {\em admissible stable cylinder} \( S \subseteq \R\) is a non empty compact subset of $\R$ with piecewise smooth boundaries 
admitting a foliation by admissible stable manifolds which stretch fully across the rectangle \( \R \) and which is the closure of its 
interior points. {\em Admissible unstable cylinder} \( U \subseteq \R\) are defined similarly with a foliation 
a foliation by admissible unstable manifolds stretching fully across the rectangle \( \R \).

The notion of admissible manifold is related to certain  cone fields $\K^s$, $\K^u$.
For every $z \in \R$ we define $\K^s_z \subset T_zM$ as the image under $De_x(p)$ evaluated at $p(z) = e^{-1}_x(z) \in I^n$, of the 
cone of width $\gamma$ 'centered' at $\real^s\oplus\{0\}$, that is, the set of vectors in $\real^n$ making an angle bounded by $\gamma$ with 
$\real^s\oplus\{0\}$. 
We define $\K^u_z \subset T_zM$ likewise considering a cone of width $\gamma$ 'centered' at $\{0\}\oplus\real^u$. 
Notice that admissible manifolds are exactly those smooth graph-like submanifolds whose tangent spaces rest inside stable and unstable cones.

We say that a $C^1$ diffeomorphism $g : S \to U$ between admissible cylinders is {\em hyperbolic} if it preserves the cone fields $\K^s$ and $\K^u$, that is,
$$
Dg(z)\K^u_z \subset int\,\K^u_{g(z)} \ \forall \ z \in S\quad\text{and}\quad Dg^{-1}(z)\K^s_z \subset int\,\K^s_{g^{-1}(z)}  \ \forall \ z \in U,
$$

\begin{definition}\label{definition.hyperbolic.branch}
Let $\R$ and $\subR$ be regular rectangles. If some iterate \( f^{m} \) maps an admissible stable cylinder  $S \subset \R$ diffeomorphically 
and hyperbolically 
to an admissible unstable cylinder $U \subset \subR$, we shall say that
$$
f^m : S \to U
$$
is a {\em hyperbolic branch}.
\end{definition}

The next pseudo-Markov property of coverings by Lyapunov rectangles is the key to the construction of hyperbolic $f$-invariant Cantor sets 
approximating the measure 
$\mu$ and satisfying (\ref{main.1}) in Theorem A.
 
\begin{proposition}\label{markov.covers}
For every $\delta > 0$ and for every Pesin set $\Lambda$ there exists constants $\kappa > 0$, $\lambda = \lambda(\chi) > 1$ only depending on $\chi > 0$, the lower 
bound for the Lyapunov exponents introduced in (\ref{eq:chi}), subsection \ref{subsec:pesin.theory}, and a finite covering by 
rectangles $\{\R_i = \R(p_i), p_i \in \Lambda; i = 1,  \cdots , t \}$ 
such that:
\begin {enumerate}
\item $\Lambda \subseteq \bigcup_{i=1}^t{B(p_i,\kappa)}$, \ and \ $B(p_i,\kappa) \subset int\,\R_i$;
\item $\diam(\R_i) \leq \delta$ for $i = 1 , \cdots , t$;
\item if $x \in \Lambda \cap B(p_i,\kappa)$ and $f^m(x) \in \Lambda \cap B(p_j,\kappa)$ then there exists an admissible stable cylinder 
$S_x \subset \R_i$ containing $x$ and an admissible unstable cylinder $U_{f^n(x)} \subset \R_j$ containing $f^m(x)$ such that 
$f^m: S_x \to U_{f^n(x)}$ is a hyperbolic branch with nonlinear rate of expansion bounded from below by $\lambda > 1$, that is:
\begin{equation}\label{nonlinear.expansion.1}
d_{W'_k}(f^k(w),f^k(w')) \geq \lambda^k{d_W(w,w')} \quad\text{for}\ k = 1, \cdots m, \ \forall \ w\,,\,w' \in W \cap S_x, 
\end{equation}
where $\dist_{W}$ and $\dist_{W'_k}$ is the metric induced by the Lyapunov charts on $W \in \Gamma^u(\R_i)$ and 
$W'_k = f^k(W \, \cap \, S)$ and $\Gamma^u(\R_i)$ is the set of admissible unstable manifolds in $\R_i$ and $m > 0$ is the return time of 
$x$. Similarly
\begin{equation}\label{nonlinear.expansion.2}
\dist_{W'_k}(f^{-k}(w),f^{-k}(w')) \geq \lambda^k{\dist_W(w,w')} \quad\text{for}\ k = 1, \cdots m, \ \forall \ w\,,\,w' \in W \cap U_{f^n(x)}, 
\end{equation}
for every $W \in \Gamma^s(\R_j)$ and $W'_k = f^{-k}(W \, \cap \, U_{f^n(x)})$, where $\Gamma^s(\R_j)$ is the set of admissible stable 
manifolds 
in $\R_j$;
\item $\diam(f^k(S_x)) \leq \delta$, for $0 \leq k \leq m$;
\end {enumerate}
\end{proposition}

We emphasize that the formation of hyperbolic branches occurs for every return from $\Lambda \cap B(p_i,\kappa)$ to 
$\Lambda \cap B(p_j,\kappa)$ which are not necessarily first return times.

\begin{definition}
We call $\mathcal{R} = \{R_i\}$ a $(\delta,\kappa, \lambda)$-Markov covering of $\Lambda$.
\end{definition}

See \cite[Definition S.4.15]{katok.mendoza} and \cite[Theorem S.4.16]{katok.mendoza}. 

\section{Proof of Theorem A: first step, choosing $\Omega$}\label{sec:proof.main.technical.lemma.1}

Theorem A's (\ref{main.1}) follows immediatly from next lemma by taking $\Omega_n = \Omega(\rho_n,s_n,\phi)$, where $\rho_n \downarrow 0^+$ 
and $s_n \to +\infty$ are suitable sequences.

\begin{lemma}\label{main.technical.lemma}
Let $\rho, s > 0$, $\phi$ continuous and $\mu$ an ergodic non atomic hyperbolic Borel probability. Then, there exist a hyperbolic basic set  
$$
\Omega = \Omega(\rho,s,\phi)
$$
such that:
\begin{enumerate}
\item every ergodic measure $\nu$ supported on $\Omega$ belongs to the weak-* open neighborhood ${\mathcal O}(\mu, \rho,s)$;
\item the following estimate holds
\begin{equation}\label{single.potential.estimative}
\dfrac{\rho\inf\phi}{1+\rho} + \dfrac{P_{\mu}(\phi)-3\rho}{1+\rho} \leq P(f|\Omega,\phi) \leq P_{\mu}(\phi)+2\rho + \rho\sup\phi;
\end{equation}
\item and the rate of hyperbolicity of $\Omega$ is bounded from below by a constant $\log\lambda > 0$.
\end{enumerate} 
\end{lemma}

Here $\lambda = \Lambda(\chi) > 1$ is the rate of expansion along local admissible unstable manifolds introduced in proposition \ref{markov.covers}, 
subsection \ref{subsec:pesin.theory}.

We dedicate the rest of this paper to prove Lemma \ref{main.technical.lemma}. In this section we prove item $(1)$ and $(3)$, while pressure 
estimates from $(3)$ take up the final two sections. 

Let us start recalling the following terminology:
\begin{enumerate}
\item $E \subset X$ is an {\em $(\epsilon,n)$-spanning set in $X$} if for every $x \in X$ there exists $y \in E$ such that 
$d(f^k(x),f^k(y)) \leq \epsilon$, 
for every $0 \leq k \leq n-1$;
\item $E \subset X$ is an {\em $(\epsilon,n)$-separated set in X} if for every pair of different points $x \not= y$ in $E$ it holds 
$d(f^k(x),f^k(y)) > \epsilon$ 
for some $0 \leq k \leq n-1$;
\item given an $f$-invariant Borel probability $\mu$ and a positive number $0 < \alpha < 1$, we say that $E$ is an 
$(\epsilon,n,\alpha)$-spanning set for $\mu$ if
$$
\mu\left(\bigcup_{x \in E}B(x,\epsilon,n)\right) \geq \alpha,
$$
where
$$
B(x,\epsilon,n) := \{\,y \in X \,:\, \dist(f^j(x),f^j(y)) < \epsilon, \ j = 0, \cdots , n-1\,\}.
$$
\end{enumerate}
$E$ is $(\epsilon,n)$-spanning in $X$ if and only if $M \subset \bigcup_{x \in E}B(x,\epsilon,n)$. Also notice that any maximal 
$(\epsilon,n)$-separated set in $X$ is $(\epsilon,n)$-spanning.

\begin{definition}
Let $f : X \to X$ continuous and $\mu$ and $f$-invariant Borel probability an $\phi$ continuous. We define the 
{\em measure-theoretical pressure of $\phi$ w.r.t. $\mu$} as  
\begin{equation}\label{additive.measure.theoretical.pressure}
P_{\mu}(\phi) := \lim_{\alpha \to 0^+}\lim_{\epsilon \to 0^+}\lim_{n \to +\infty}\dfrac{1}{n}\log\left(\inf_E\left\{\sum_{x \in E}\exp{S_n\phi(x)}\right\}\right),
\end{equation}
where $\mu$ is an $f$-invariant Borel probability and the \emph{infimum taken over $(\epsilon,n,\alpha)$-spanning subsets $E \subset M$}.
\end{definition}

\begin{proposition}\label{free.energy.prop}
Let $f : X \to X$ a continuous self map of a compact metric space $(X,d)$, $\phi$ continuous and $\mu \in {\mathcal M}_f$ an ergodic 
$f$-invariant Borel probability. Then, for every $0 < \alpha < 1$,
\begin{equation}\label{additive.measure.theoretical.pressure.1}
P_{\mu}(\phi) = \lim_{\epsilon \to 0^+}\lim_{n \to +\infty}\dfrac{1}{n}\log\left(\inf_E\left\{\sum_{x \in E}\exp{S_n\phi(x)}\right\}\right) = h(\mu) + \int\phi{d\mu},
\end{equation}
\emph{infimum taken over $(\epsilon,n,\alpha)$-spanning subsets $E \subset M$}.
\end{proposition}
See \cite[Theorem 1.1]{mendoza.1988}.

The construction of $\Omega$ will rely upon proposition \ref{free.energy.prop}.

For this we need to fix $\alpha > 0$, $\delta > 0$, $n > 0$ and a finite $(\delta,n,\alpha)$-spanning subset $E_0$ such that each 
$x \in E_0$ is endowed with a hyperbolic branch $f^{R(x)} : S_x \to U_{f^n(x)}$ for a suitable return time to a hyperbolic Pesin set of 
quasi-generic points $\Lambda_0$. Then we choose a suitable subset of those hyperbolic branches to generate a horseshoe with finitely many 
branches and variable return times $\Omega^*$ and then we prove that $\Omega = \bigcup_{n \in \enteros}f^n(\Omega^*)$, the 
$f$-invariant saturate of $\Omega^*$ satisfies the inequalities (\ref{single.potential.estimative}) in Lemma \ref{main.technical.lemma}.
\ 
\\
\\
{\bf First we fix once for all $\rho > 0$ and $s > 0$ and $\{\psi_i\}$ a countable dense subset of continuous functions}.
\ 
\\
\\
Then, we choose a hyperbolic Pesin set of quasi-generic points. For this we fix a hyperbolic Pesin set $\Lambda$ and define
\begin{equation}\label{quasi.generic}
\Lambda_N := \{ x \in \Lambda: \left|\sum_{k=0}^{n-1}\psi_i(f^k(x)) - \int\psi_i{d\mu} \right| < \rho/2 \quad\forall i \leq s \quad\forall n \geq N\}
\end{equation}
Then we pick up a large integer $N_0 > 0$ such that $\Lambda_0 := \Lambda_{N_0}$ has
\begin{equation}\label{definition.quasigeneric.hyperbolic.set}
 \mu(\Lambda_0) \geq \dfrac{\mu(\Lambda)}{2}
\end{equation}
This is possible since,
$$
{\mathcal G}_{\rho,s,N} = \{ x \in M: \left|\sum_{k=0}^{n-1}\psi_i(f^k(x)) - \int\psi_i{d\mu} \right| < \rho/2 \quad\forall i \leq s \quad\forall n \geq N\}
$$
is an increasing sequence and $\mu({\mathcal G}_{\rho,s,N}) \uparrow 1$ when $N \to +\infty$.
\ 
\\
\\
{\bf Then fix $\alpha > 0$ defining}
\begin{equation}\label{defining.alpha}
\alpha := \dfrac{\mu(\Lambda)}{4}
\end{equation}

Next step is to {\bf fix a small precision $\delta > 0$}:

\begin{lemma}
There exists $\delta(\rho,s) > 0$ such that, for every $0 < \delta < \delta(\rho,s)$ it holds
\begin{equation}\label{definition.delta.1}
 \forall \ x,y \in M: \quad d(x,y) < \delta \Longrightarrow |\psi_{i}(x)-\psi_{i}(y)|  <  \rho/2, \quad\quad \forall \ i \leq s,
\end{equation}
\begin{equation}\label{definition.delta.2}
 \forall \ x,y \in M: \quad d(x,y) < \delta \Longrightarrow |\phi(x)-\phi(y)| < \rho,
\end{equation}
and
\begin{equation}\label{definition.delta.3}
\left|\lim_{n \to +\infty}\dfrac{1}{n}\log\left(\inf_E\left\{\sum_{x \in E}\exp{S_n\phi}\right\}\right) - P_{\mu}(\phi)\right| < \rho/4.
\end{equation}
infimum is taken over all the $(\delta,n,\alpha)$-spanning subsets $E$. 
\end{lemma}

(\ref{definition.delta.1}) and (\ref{definition.delta.2}) follows from the continuity of $\psi_i$ and $\phi$; 
(\ref{definition.delta.3}) follows from the definition of the limit (\ref{additive.measure.theoretical.pressure.1}).

\ 
\\
\\
{\bf Choosing a large time $n \geq N_0 > 0$}
\ 
\\
Now we fix a $(\delta/4,\kappa, \lambda)$-Markov covering of $\Lambda$ and choose $N_0 > 0$ in the definition of $\Lambda_0$ 
larger if necessary such that,
\begin{equation}\label{definition.n.1}
\forall \ n \geq N_0: \quad \left|\dfrac{1}{n}\log\left(\inf_E\left\{\sum_{x \in E}\exp{S_n\phi}\right\}\right) - P_{\mu}(\phi)\right| < \rho/2,
\end{equation}
and
\begin{equation}\label{definition.n.2}
\forall \ n \geq N_0: \quad \exp(n\rho) \geq \#{\mathcal R}.
\end{equation}
Moreover, we choose $N_0$ sufficiently large such that for every $n \geq N_0$ a large portion of points in $\Lambda_0$ return within a 
time $R(x) \in [n,(1+\rho)n]$ giving rise to quasi-generic branches. This is the content of the following

\begin{lemma}\label{return.lemma}
There exists a large $N_0 > 0$ satisfying (\ref{definition.quasigeneric.hyperbolic.set}), (\ref{definition.n.1}) and 
(\ref{definition.n.2}) with the following property: for every $n \geq N_0$ and for every open ball $B(p_i,\kappa) \subset \R_i$ of 
the $(\delta/4,\kappa, \lambda)$-Markov covering of $\Lambda$ there exists a subset 
$\Lambda_{0,i} \subset B(p_i,\kappa) \cap \Lambda_0$ with
$$
\mu(\Lambda_{0,i}) \geq \mu(B(p_i,\kappa) \cap \Lambda_0)/2  
$$
such that for every $x \in \Lambda_{0,i}$ returns to $B(p_i,\kappa) \cap \Lambda_0$ with a return time
\begin{equation}\label{return.time}
R(x) \in [n,(1+\rho)n].
\end{equation}
\end{lemma}
\begin{remark}
 We underline that $R(x)$ is not necessarily the first return time of $x$.
\end{remark}
\begin{proof}
This follows from the ergodicity of $\mu$. Cf. \cite{katok.mendoza}. 
Let $A \subset M$ be a Borel set with $\mu(A) > 0$. Given $\rho > 0$ and $n > 0$ define
$$
A_{\rho,n} := \{x \in A: x \ \text{return to $A$ with return time}\ R(x) \in [n,(1+\rho)n]\}
$$
Then given $0 < \epsilon < 1$ there exists $N > 0$ such that
$$
\mu(A_{\rho,n}) \geq (1-\epsilon)\mu(A) \quad\text{for every} \ n \geq N.
$$
Then we apply this lemma to $A = B(p_i,\kappa) \cap \Lambda_0$ and $\epsilon = 1/2$.
\end{proof}
\ 
\\
\\
{\bf We fix once for all some $n \geq N_0$ satisfying (\ref{definition.quasigeneric.hyperbolic.set}), 
(\ref{definition.n.1}), (\ref{definition.n.2}) and the return time property in Lemma \ref{return.lemma}.}
\ 
\\
\\
{\bf Choosing $E_0$}
\ 
\\
\\
Notice that,
$$
\mu(\bigcup_{i}\Lambda_{0,i}) \geq \alpha.
$$
Therefore we can choose a maximal $(\delta,n)$ separated subset $E_0 \subset \bigcup_i\Lambda_{0,i}$ such that
\begin{equation}\label{main.estimative.1}
\left|\dfrac{1}{n}\log\left(\sum_{x \in E_0}\exp(S_n\phi(x))\right) - P_{\mu}(\phi)\right| < \rho,
\end{equation}
we can do this by (\ref{definition.n.1}).
\ 
\\
\\
{\bf The construction of $\Omega$}
\ 
\\
\\
By construction for each point $x \in E_0$ there exists a hyperbolic branch $f^{R(x)}: S_x \to U_{f^{R(x)(x)}}$ contained in some $\R_i$ and 
such that
\begin{equation}\label{rho.shadowing.2}
\diam(f^j(S_x)) < \delta/4 \quad\text{for every} \ j = 0, \cdots , R(x)-1.
\end{equation}
This and the condition of separation of points in $E_0$ implies that any two different branches subordinated to the same rectangle are 
disjoint.

Then we choose $\ell > 0$ and a subset 
$$
E_{\ell} := B(p_{\ell},\kappa) \cap E_0
$$
such that
\begin{equation}\label{main.estimative.2}
\sum_{x \in E_{\ell}}\exp{S_n\phi(x)} \geq \sum_{x \in E_{\ell'}}\exp{S_n\phi(x)} \quad\text{for every}\quad \ell' \not= \ell,
\end{equation}
and define $\Omega(\rho,s,\phi)$ as the $f$-invariant saturate of the horseshoe with finitely many branches defined by the collection of 
branches 
$\{f^R(x): S_x \to U_x: x \in E_{\ell}\}$ chosen by condition (\ref{main.estimative.2}):
\begin{equation}\label{definition.Omega}
\Omega(\rho,s,\phi) = \bigcup_{n \in \enteros}f^n\left(\bigcap_{k \in \enteros}(f^{R})^{k}\bigcup_{x \in E_{\ell}}S_x\right), 
\end{equation}
where $f^R | S_x = f^{R(x)}$.

\begin{lemma}\label{strong.approximation.lemma}
$\nu \in {\mathcal O}(\mu,\rho,s)$ for every ergodic $f$-invariant Borel probability $\nu$ supported on $\Omega$. Moreover, the rate of 
hyperbolicity of $\Omega$ is bounded from below by $\log\lambda > 0$.
\end{lemma}
\begin{proof}
This follows from Proposition \ref{strong.approximation.property} since the branches $\{f^{R(x)} : S_x \to U_x: x \in E_{\ell}\}$ are 
$(\rho,s)$-quasi generic, that is,
\begin{equation}\label{rho-s.generic}
\forall \ y \in S_x: \quad \left|\frac 1n\sum_{j=0}^{R(x)-1}\psi_{i}(f^{j}(y))-\int\psi_{i}d\mu\right|\leq \rho \quad \forall i\leq s.
\end{equation}
Indeed, for every $i$ every return from $x \in \Lambda_{0,i}$ to $f^{R(x)}(x) \in \Lambda_{0,i}$ giving rise to a 
$(\rho,s)$-quasi-generic branch (see Definition \ref{quasigeneric.definition})
$$
f^{R(x)} : S  \to U, \quad\text{with}\quad S,U \subset \R_i.
$$
This follows from (\ref{definition.delta.1}) in the definition of $\delta$, since $\diam(f^j(S)) < \delta/4$ for 
$j = 0, \cdots , R(x)-1$ by the definition of a pseudo-Markov covering and $R(x) \geq n \geq N_0$, therefore for every $y \in S$
\begin{eqnarray*}
\left|\sum_{k=0}^{n-1}\psi_i(f^k(y)) - \int\psi_i{d\mu} \right| & \leq & \left|\sum_{k=0}^{n-1}\psi_i(f^k(y)) - \sum_{k=0}^{n-1}\psi_i(f^k(x))\right| + \left|\sum_{k=0}^{n-1}\psi_i(f^k(x)) - \int\psi_i{d\mu} \right|\\
                                                                                       & \leq & \rho/2 + \rho/2\\
                                                                                       &  =   & \rho
\end{eqnarray*}
The bound on the rate of hyperbolicity follows from nonlinear expansion property (\ref{nonlinear.expansion.1}) and 
(\ref{nonlinear.expansion.2}) in proposition \ref{markov.covers}.
\end{proof}

\section{Second step: counting periodic orbits}\label{sec:proof.main.technical.lemma.2}

To estimate the topological pressure of $\Omega(\rho,s,\phi)$ we need to estimate the cardinality of periodic orbits. 
Let $Per(N)$ denote the set of periodic points with prime period $N$: $x \in Per(N)$ iff $f^N(x) = x$ and $f^k(x) \not= x$ for 
$0 < k < N$.

\begin{lemma}\label{pressure.hyperbolic.sets}
Let $\Omega$ be a basic set for a $C^r$ ($r \geq 1$) diffeomorphism and 
$\phi$ continuous. Then:
\begin{equation}\label{topological.pressure.hyperbolic}
P(f|\Omega,\phi) = \limsup\limits_{N \to +\infty}\dfrac{1}{N}\log\left(\sum\limits_{x \in Per(N)}\exp{S_N\phi}\right).
\end{equation}
\end{lemma}

This was proved in \cite[Section 7.19 (7.11)]{ruelle.2004} for Smale spaces \cite[Section 7.1]{ruelle.2004} which includes basic sets of $C^r$ 
($r \geq 1$) diffeomorphisms and bilateral full shifts of finitely many symbols.

To estimate the limsup in (\ref{topological.pressure.hyperbolic}) one has to keep track of the combinatorics of periodic orbits. 
This is done as follows.

Let $E^p = E \times \cdots \times E$ the cartesian product of $p$-copies of $E$, where we denote $E = E_{\ell}$ to simplify notation.

\begin{lemma}\label{symbolic.dynamics}
$f^{R} | \Omega^*$ is topologically conjugated to the full-shift on $\#E$ symbols.
\end{lemma}

\begin{proof}
We first observe that, as the stable cylinders $S_x$, $x \in E$ are disjoint and
$$
\Omega^* = \bigcap\limits_{n \in \enteros}(f^{R})^{(n)}\left(\bigcup\limits_{x \in E}S_x\right),
$$
where $x_n = (f^{R})^{(n)}(x)$ is defined inductively as
$$
\begin{cases}
x_{n+1} = f^{R(x_n)}(x_n) & \text{for} \ n \geq 0\\
x_{n-1} = f^{-R(x_n)}(x_n) & \text{for} \ n \leq 0
\end{cases}
$$
setting $x = x_0$. That is,
$$
\begin{cases}
(f^{R})^{(n)}(x) = f^{\sum_{0 \leq i < n}R(x_i)}(x) & \text{for} \ n \geq 0\\
(f^{R})^{(n)}(x) = f^{-\sum_{n < i \leq 0}R(x_i)}(x) & \text{for} \ n \leq 0
\end{cases}
$$
In particular, for every $z \in \Omega^*$ there exists a unique $x \in E$ such that
$$
\dist(f^j(x),f^j(z)) < \delta/4 \quad\text{for every}\quad j = 0, \cdots , R(x)-1.
$$
Unicity of $x \in E$ follows since $E$ is part of a $\delta/2$-separated set. Then we associate to every $x \in \Omega^*$ a unique bi-infinite sequence 
$\{x_n\} \in E^{\enteros}$ defined by shadowing the orbit of $z$. This shows that $f^{R} | \Omega^*$ is topologically 
conjugated to the full-shift on $\#E$ symbols.
\end{proof}

\begin{corollary}\label{coding.periodic.orbits}
If $z \in \Omega$ is $f$-periodic, and assuming without loss of generality that $z \in \Omega^*$, then its successive returns to $\Omega^*$ 
define an $f^{R}$-periodic orbit given by a uniquely defined sequence $x_k \in E$, $0 \leq k < p$, $p > 1$, such that
$$
x_{k+1} = f^{R}(x_k), \ \text{for} \ 0 \leq k < p \ \text{and} \ x_{0} = f^{R}(x_{p-1}).
$$
Moreover,
\begin{equation}\label{delta.shadowing.condition}
\dist(f^{j+\sum_{i < k}R(x_i)}(z), f^j(x_k)) < \delta/4 \quad\text{for}\quad j = 0, \cdots , R(x_k)-1, \ k = 0, \cdots p-1.
\end{equation} 
\end{corollary}

\begin{remark}
Let $z \in Per(f|\Omega)$ a periodic point. Then according to our previous discussion its prime period $N = N(z)$ is a linear combination of the 
basic periods $n_1, \cdots , n_{\#E}$ of the branches generating $\Omega^*$, namely, there exists integers $p_i \in \enteros^+$, $i = 1, \cdots \#E$ 
such that
\begin{equation}\label{admissible.periods.equation.1}
N = p_1n_1 + \cdots p_{\#E}n_{\#E}.
\end{equation}
\end{remark}

\begin{definition}
We say that $N \in \enteros^+$ is an admissible period if it satisfies (\ref{admissible.periods.equation.1}) for some sequence of non-negative 
integers $p_i \in \enteros^+$, $i = 1, \cdots \#E$.
\end{definition}

\begin{lemma}\label{admissible.periods}
Let $N > 0$ be an admissible period, $z \in Per(N)$ and $[x_0, \cdots , x_{p-1}] \in E^p$ the unique sequence of points in $E$ which 
sucessively $\delta/4$-shadows ${\mathcal O}(z)$ when the orbit cycles around $\Omega$ provided by Corollary \ref{coding.periodic.orbits}. 
Then
\begin{equation}\label{counting.cycles}
\dfrac{N}{n(1+\rho)} \leq p \leq \dfrac{N}{n},
\end{equation}
where $n >0$ were fixed after lemma \ref{return.lemma}.
\end{lemma}
\begin{proof}
Let $N = N(z)$ the prime period of $z$ then 
\begin{equation}\label{admissible.periods.equation.0}
N = \sum\limits_{i=0}^{p-1}R(x_i). 
\end{equation}
Using (\ref{admissible.periods.equation.0}) and that $n \leq R(x_k) \leq (1+\rho)n$ for every $k = 0, \cdots , p-1$, one conclude that 
$pn \leq N \leq (1+\rho)np$ and we get (\ref{counting.cycles}). 
\end{proof}

\begin{lemma}\label{shadowing.argument.2}
Let $N$ be an admissible period, $z \in Per(N)$ a periodic point for $f | \Omega$ of prime period $N$ and 
$$
[x_0, \cdots , x_{p-1}] \in E^p
$$
the encoding sequence defined in Corollary \ref{coding.periodic.orbits}. Then,
\begin{equation}\label{balance.equation.1}
\exp\left(S_N(\phi+\rho)(z)\right) \geq \prod\limits_{k=0}^{p-1}\exp\left(S_{R(x_k)}\phi(x_k)\right)
\end{equation}
and
\begin{equation}\label{balance.equation.2}
\exp\left(S_N(\phi-\rho)(z)\right) \leq \prod\limits_{i=0}^{p-1}\exp\left(S_{R(x_k)}\phi(x_k)\right).
\end{equation}
\end{lemma}
\begin{proof}
Recall that the branches originating the Alekseev set $\Omega^*$ satisfy
$$
\left|\sum_{j=0}^{R(x)-1}\phi(f^j(y)) - \sum_{j=0}^{R(x)-1}\phi(f^j(z))\right| < R(x)\rho \quad\text{for every}\quad y,z \in S_x.
$$
This is by our choice of $\delta > 0$ since $\diam(f^j(S_x)) < \delta/4$ for $j = 0, \cdots , R(x)-1$.
Then by (\ref{admissible.periods.equation.0}) and (\ref{delta.shadowing.condition})
$$
\left|\sum_{j=0}^{N-1}\phi(f^j(z)) - \sum\limits_{k=0}^{p-1}\sum\limits_{j=0}^{R(x_k)-1}\phi(f^j(x_k))\right| < N\rho.
$$
Therefore
$$
S_N(\phi+\rho)(z) \geq \sum\limits_{k=0}^{p-1}\sum\limits_{j=0}^{R(x_k)-1}\phi(f^j(x_k)) \quad\text{and}\quad S_N(\phi-\rho)(z) \leq \sum\limits_{k=0}^{p-1}\sum\limits_{j=0}^{R(x_k)-1}\phi(f^j(x_k)).
$$
Taking the exponential at both sides we get (\ref{balance.equation.1}) and (\ref{balance.equation.2}).
\end{proof}

\section{Third step: pressure estimates and conclusion}\label{sec:proof.main.technical.lemma.3}

To finish the proof we need to estimate the pressure of the $f$-invariant saturate of the Alekseev set chosen in previous 
step.

\begin{definition}
Let $ p > 0$ be any positive integer. We denote $\Delta(p)$ the set of admissible periods of periodic orbits in $\Omega(\rho,s,\phi)$, 
encoded into $E^p$:
\begin{equation}\label{admissible.periods&coding}
\Delta(p) := \{N \,:\, \exists \ [x_0, \cdots , x_{p-1}] \in E^p \quad\text{such that}\quad N = \sum\limits_{k=0}^{p-1}R(x_k)\,\}.
\end{equation}
\end{definition}

\begin{lemma}\label{main.bounds}
For every positive integer $p > 0$ it holds: 
\begin{eqnarray}
\label{1}
\sum_{N \in \Delta(p)}\sum\limits_{z \in Per(N)}\exp(S_N(\phi+\rho)(z)) \geq \left[\sum_{x \in E}\exp(S_{R(x)}\phi(x))\right]^{p} \\
\label{2}
\sum_{N \in \Delta(p)}\sum\limits_{z \in Per(N)}\exp(S_N(\phi-\rho)(z)) \leq \left[\sum_{x \in E}\exp(S_{R(x)}\phi(x))\right]^{p} 
\end{eqnarray}
\end{lemma}
\begin{proof}
Using inequality (\ref{balance.equation.1}) in Lemma \ref{shadowing.argument.2}, (\ref{counting.cycles}) and the identity
$$
(a_1 + \cdots + a_n)^m = \sum\limits_{[i_1, \cdots , i_m]}\,a_{i_1} \cdots a_{i_m} \quad\text{where}\quad [i_1, \cdots , i_m] \in \{1, \cdots ,n\}^m
$$
we get
\begin{eqnarray*}
\sum_{N \in \Delta(p)}\sum\limits_{z \in Per(N)}\exp(S_N(\phi+\rho)(z)) & \geq & \sum\limits_{[x_1, \cdots , x_p] \in E^p} \ \prod\limits_{k=0}^{p-1}\exp\left(S_{R(x_k)}\phi(x_k)\right) \\
                                                  &  =   & \left[\sum_{x \in E}\exp(S_{R(x)}\phi(x))\right]^{p},
\end{eqnarray*}
thus proving (\ref{1}). Inequality (\ref{2}) follows similarly using inequality (\ref{balance.equation.2}) in Lemma \ref{shadowing.argument.2}.
\end{proof}

\medskip
\ 
\\
{\bf Proof of Lemma \ref{main.technical.lemma}}
\ 
\\
\\
We proved already that $\nu \in \mathcal{O}(s,3\rho)$ for every ergodic $f$-invariant Borel probability $\nu$ supported on $\Omega$ in Lemma 
\ref{strong.approximation.lemma}. Now we want to prove the estimates (\ref{single.potential.estimative}) in Lemma \ref{main.technical.lemma}. For this 
we use previous section results to estimate $P(f | \Omega, \phi)$.
  
We first notice that, as $R(x) \in [n,(1+\rho)n]$ for every $x \in E$ we have
\begin{equation}\label{3}
\sum_{x \in E}\exp(S_{R(x)}\phi(x)) \geq \sum_{x \in E}\exp(S_{n}\phi(x)) \times \exp(n\rho\inf\phi)
\end{equation}
and
\begin{equation}\label{4}
\sum_{x \in E}\exp(S_{R(x)}\phi(x)) \leq \sum_{x \in E}\exp(S_{n}\phi(x)) \times \exp(n\rho\sup\phi).
\end{equation}
By (\ref{main.estimative.1}) and the choice of $E = E_{\ell}$ as the set which maximizes the sums 
$$
\sum_{x \in E_{\ell}}\exp(S_{n}\phi(x))
$$
in (\ref{main.estimative.2}), we have
$$
\#{\mathcal R}\sum_{x \in E}\exp(S_{n}\phi(x)) \geq \sum_{x \in E_0}\exp(S_{n}\phi(x)) >  \exp(n[P_{\mu}(\phi)-\rho])
$$
thus giving
\begin{equation}\label{5.1}
\sum_{x \in E}\exp(S_{n}\phi(x)) \geq \exp(n[P_{\mu}(\phi)-2\rho]),
\end{equation}
since $\exp(n\rho) > \#{\mathcal R}$, by (\ref{definition.n.2}). On the other hand, as $E \subset E_0$
\begin{equation}\label{5.2}
\sum_{x \in E}\exp(S_{n}\phi(x)) \leq \sum_{x \in E_0}\exp(S_{n}\phi(x)) < \exp(n[P_{\mu}(\phi)+\rho]).
\end{equation}

Therefore, substituting (\ref{5.1}) into (\ref{3}) and recalling (\ref{1}) we get the lower bound
\begin{equation}\label{6}
\sum_{N \in \Delta(p)}\sum\limits_{z \in Per(N)}\exp(S_N(\phi+\rho)(z)) \geq \left[\exp(n[P_{\mu}(\phi)-2\rho])\times \exp(n\rho\inf\phi)\right]^{p}
\end{equation}

On the other hand, as $R(x) \in [n,(1+\rho)n]$ we have
\begin{equation}\label{bounds.admissible.periods}
np \leq N \leq n(1+\rho)p \quad\text{for every admissible period}\quad N \in \Delta(p). 
\end{equation}
Therefore, 
\begin{equation}\label{cardinal.admissible.periods}
\#\Delta(p) \leq n(1+\rho)p - np = np\rho 
\end{equation}

Hence, maximizing the sums $\sum_{z \in Per(N)}\exp(S_N(\phi+\rho)(z))$ over the set of admissible periods $N \in \Delta(p)$ we get
$$
\#\Delta(p)\sum\limits_{z \in Per(N_p)}\exp(S_{N_p}(\phi+\rho)(z)) \geq \left[\exp(n[P_{\mu}(\phi)-2\rho])\times \exp(n\rho\inf\phi)\right]^{\frac{N_p}{(1+\rho)n}},
$$
for a suitable admissible period $N_p \in \Delta(p)$, where we have used (\ref{counting.cycles}) to bound from below $p > 0$ in terms of $N_p$.

Therefore by (\ref{cardinal.admissible.periods}) we get
\begin{equation}\label{8}
np\rho\times\sum\limits_{z \in Per(N_p)}\exp(S_{N_p}(\phi+\rho)(z)) \geq \left[\exp(n[P_{\mu}(\phi)-2\rho])\times \exp(n\rho\inf\phi)\right]^{\frac{N_p}{(1+\rho)n}}.
\end{equation}

Then, taking logarithms and dividing by $N_p$ in (\ref{8}) and letting $p \to +\infty$, we get a sequence $N_p \to +\infty$ of admissible 
periods, such that, 
\begin{eqnarray*}
P(f|\Omega,\phi+\rho) & = & \limsup\limits_{N \to +\infty}\dfrac{1}{N}\log\sum\limits_{z \in Per(N)}\exp(S_{N}(\phi+\rho)(z))\\
                      & \geq & \lim\limits_{p \to +\infty}\dfrac{\log(np\rho)}{N_p} + \lim\limits_{p \to +\infty}\dfrac{1}{N_p}\log\sum\limits_{z \in Per(N_p)}\exp(S_{N_p}(\phi+\rho)(z))\\
                      & \geq & \dfrac{P_{\mu}(\phi)-2\rho}{1+\rho} + \dfrac{\rho\inf\phi}{1+\rho}.
\end{eqnarray*}

On the other hand, introducing (\ref{5.2}) into (\ref{4}) and recalling (\ref{2}) we get 
\begin{equation}\label{7}
\sum_{N \in \Delta(p)}\sum\limits_{z \in Per(N)}\exp(S_N(\phi-\rho)(z)) \leq \left[\exp(n[P_{\mu}(\phi)+\rho])\times \exp(n\rho\sup\phi)\right]^{p}
\end{equation}
Therefore, for every admissible period $N \in \Delta(p)$ and for every $p \geq 0$,
\begin{equation}\label{7.1}
\sum\limits_{z \in Per(N)}\exp(S_N(\phi-\rho)(z)) \leq \left[\exp(n[P_{\mu}(\phi)+\rho])\times \exp(n\rho\sup\phi)\right]^{p}
\end{equation}
and then, taking logarithms and dividing by $N$, for every admissible period $N \in \Delta(p)$, $p \geq 0$,
\begin{eqnarray*}
\dfrac{1}{N}\log\left(\sum\limits_{z \in Per(N)}\exp(S_N(\phi-\rho)(z))\right) & \leq & \dfrac{np}{N}(P_{\mu}(\phi)+\rho + \rho\sup\phi)\\
                                                                               & \leq & P_{\mu}(\phi)+\rho + \rho\sup\phi,
\end{eqnarray*}
using (\ref{bounds.admissible.periods}). Hence,
\begin{eqnarray*}
P(f|\Omega,\phi-\rho) & = &  \limsup\limits_{N \to +\infty}\dfrac{1}{N}\log\left(\sum\limits_{z \in Per(N)}\exp(S_N(\phi-\rho)(z))\right)\\
                      & \leq & P_{\mu}(\phi)+\rho + \rho\sup\phi.
\end{eqnarray*}
This proves estimative (\ref{single.potential.estimative}) at Lemma \ref{main.technical.lemma}, after a straightforward calculation, using $P(\phi + c) = P(\phi)$ ([Theorem 2.1, (vii)]\cite{walters}). \ {\bf QED}

\end{document}